\documentclass[11pt]{article}
\let\astorig\ast
\usepackage{graphicx,algorithmic,algorithm}%
\usepackage{enumerate}
\usepackage{tikz}
\usepackage[nomessages]{fp}%
\usetikzlibrary{shapes}
\usepackage{hyperref}
\hypersetup{
    pdftitle=   {Obstructions for bounded shrub-depth and rank-depth},
   pdfauthor=  {O-joung Kwon and Rose McCarty and Sang-il Oum and Paul Wollan}
}
\usepackage{todonotes}
\usepackage{authblk}
\usepackage{amsthm,amssymb,amsmath}
\usepackage{mathabx}
\newtheorem{theorem}{Theorem}[section]
\newtheorem{lemma}[theorem]{Lemma}
\newtheorem{proposition}[theorem]{Proposition}
\newtheorem{corollary}[theorem]{Corollary}

\newtheorem{question}{Question}

\newtheorem*{thm:main1}{Theorem~\ref{thm:main1}}
\newtheorem*{thm:main2}{Theorem~\ref{thm:main2}}

\newcommand\abs[1]{\lvert #1\rvert}

\newcommand{\pivot}{\wedge}

\newcommand\tri{\boxslash}

\def\K_#1{{K_{#1}}}
\def\S_#1{\overline{K_{#1}}}

\newcommand\cutrk{\rho}

\newcommand\f{f}

\begin{document}
\title{Obstructions for bounded shrub-depth and rank-depth}

\author[1,2]{O-joung Kwon\thanks{Supported by IBS-R029-C1.}\thanks{Supported by the National Research Foundation of Korea (NRF) grant funded by the Ministry of Education (No. NRF-2018R1D1A1B07050294).}}

\affil[1]{\small Department of Mathematics, Incheon~National~University, Incheon,~South~Korea.}
\affil[2]{\small Discrete Mathematics Group, Institute~for~Basic~Science~(IBS), Daejeon,~South~Korea.}
\author[3]{Rose McCarty}
\affil[3]{\small Department of Combinatorics and Optimization, University~of~Waterloo, Waterloo,~Canada.}

\author[$\astorig$2,4]{Sang-il Oum}

\affil[4]{\small Department of Mathematical Sciences, KAIST,  Daejeon,~South~Korea.}

\author[5]{Paul~Wollan}
\affil[5]{\small Department of Computer Science, University~of~Rome,~``La~Sapienza'', Rome,~Italy.}

\date\today
\maketitle
  \setcounter{footnote}{1}%
  \footnotetext{E-mail addresses: \texttt{ojoungkwon@gmail.com} (Kwon),
    \texttt{rose.mccarty@uwaterloo.ca} (McCarty), 
    \texttt{sangil@ibs.re.kr} (Oum),
    \texttt{wollan@di.uniroma1.it} (Wollan).
  }

\begin{abstract}
  Shrub-depth and rank-depth are dense analogues of the tree-depth of a graph.
  It is well known that a graph has large tree-depth if and only if it has a long path as a subgraph. We prove an analogous statement for shrub-depth and rank-depth,
  which was conjectured by Hlin{\v{e}}n{\'y}, Kwon, Obdr{\v{z}}{\'a}lek, and Ordyniak [Tree-depth and vertex-minors, European J.~Combin. 2016].
  Namely, we prove that a graph has large rank-depth if and only if it has a vertex-minor isomorphic to a long path.
  This implies that for every integer $t$, the class of graphs with no vertex-minor isomorphic to the path on $t$ vertices
  has bounded shrub-depth.
 \end{abstract}

\section{Introduction}
Ne\v{s}et\v{r}il and Ossona de Mendez~\cite{NO2006} introduced the \emph{tree-depth} of a graph $G$, which is defined as the minimum height of a rooted forest whose closure contains the graph $G$ as a subgraph. This concept has been proved to be very useful,
in particular in the study of graph classes of bounded expansion~\cite{NO2012}.
Similar to the grid theorem for tree-width of Robertson and Seymour~\cite{RS1991},
it is known that a graph has large tree-depth if and only if it has a long path as a subgraph, see \cite[Proposition 6.1]{NO2006}.
For more information on tree-depth, the readers are referred to the surveys \cite{NO2015,NO2006} by Ne\v{s}et\v{r}il and Ossona de Mendez.

There have been attempts to define an analogous concept suitable for dense graphs.
For tree-width, this line of research has
resulted in width parameters such as clique-width~\cite{CO2000} and rank-width~\cite{OS2004}. 
In a conference paper published in 2012, Ganian, Hlin\v{e}n\'y, Ne\v{s}et\v{r}il, Obdr\v{z}\'alek, Ossona de Mendez, and Ramadurai~\cite{GHNOOR2012} introduced the \emph{shrub-depth} of a graph class, as an extension of tree-depth for dense graphs.
Recently, DeVos, Kwon, and Oum~\cite{DKO2019} introduced the \emph{rank-depth} of a graph
as an alternative to shrub-depth and showed that shrub-depth and rank-depth are equivalent in the following sense.
\begin{theorem}[DeVos, Kwon, and Oum~\cite{DKO2019}]\label{thm:rdsd}
A class of graphs has bounded rank-depth if and only if it has bounded shrub-depth.
\end{theorem}
Theorem~\ref{thm:rdsd} allows us to work exclusively with rank-depth going forward, and we omit the definition of shrub-depth. The definition of rank-depth is presented in Section~\ref{sec:prelim}.

One useful feature of rank-depth is that it does not increase under taking vertex-minors.
In other words, if $H$ is a vertex-minor of $G$, then the rank-depth of $H$ is at most that of $G$. This allows us to consider obstructions for having small rank-depth in terms of vertex-minors.
DeVos, Kwon, and Oum~\cite{DKO2019} showed that the rank-depth of the $n$-vertex path is larger than $\log n / \log (1+4\log n)$ for $n\ge 2$
and thus graphs having a long path as a vertex-minor have large rank-depth. Hlin{\v{e}}n{\'y}, Kwon, Obdr{\v{z}}{\'a}lek, and Ordyniak~\cite{HlinenyKJS2016}
conjectured that the converse is also true. Their original conjecture was stated in terms of shrub-depth but is equivalent by Theorem~\ref{thm:rdsd}.
We prove their conjecture as follows.
\begin{theorem}\label{thm:main1}
	For every positive integer $t$, there exists an integer~$N(t)$ such that every graph of rank-depth at least $N(t)$ contains a vertex-minor isomorphic to the path on $t$ vertices.
\end{theorem}
Courcelle and Oum~\cite{CO2004} showed that there is a CMSO$_1$ transduction that maps a graph to its vertex-minors.
Therefore, Theorem~\ref{thm:main1} implies that a class~$\mathcal G$ of graphs has bounded rank-depth if and only if for every CMSO$_1$ transduction $\tau$, there exists an integer $t$ such that $P_t\notin \tau(\mathcal G)$, which was conjectured by Ganian, Hlin\v{e}n\'y, Ne\v{s}et\v{r}il, Obdr\v{z}\'alek, and Ossona de Mendez~\cite{GHNOO2017}.

If we apply the same proof for bipartite graphs, then we prove the following theorem on pivot-minors of graphs. Pivot-minors are more restricted in a sense that every pivot-minor of a graph is a vertex-minor but not every vertex-minor is a pivot-minor.
This theorem allows us to deduce a corollary for binary matroids of large branch-depth.
\begin{theorem}\label{thm:main2}
	For every positive integer $t$, there exists an integer~$N(t)$ such that every bipartite graph of rank-depth at least $N(t)$ contains a pivot-minor isomorphic to $P_t$.
  \end{theorem}

The paper is organized as follows.
In Section~\ref{sec:prelim}, we review vertex-minors and rank-depth and prove a few useful properties related to rank-depth.
In Section~\ref{sec:proof}, we present the proof of Theorem~\ref{thm:main1}.
In Section~\ref{sec:pivotminor}, we obtain Theorem~\ref{thm:main2} and discuss its consequence to binary matroids of large branch-depth.
Finally, in Section~\ref{sec:remark} we conclude the paper by giving some remarks
on linear $\chi$-boundedness of graphs with no $P_t$ vertex-minors.

\section{Preliminaries and basic lemmas}\label{sec:prelim}
All graphs in this paper are simple, meaning that neither loops nor parallel edges are allowed. 
For two sets $X$ and $Y$, we write $X\Delta Y$ for $(X\setminus Y)\cup (Y\setminus X)$.

Let $G$ be a graph. We write $V(G)$ and $E(G)$ for the vertex set and the edge set of $G$, respectively.
For a vertex $v$ of $G$, we write $N_G(v)$ to denote the set of all neighbors of $v$ in $G$.
For a vertex $v$ of $G$, let $G-v$ denote the graph obtained from $G$ by removing $v$ and all edges incident with $v$.
For an edge $e$ of $G$, let $G-e$ denote the graph obtained from $G$ by removing $e$.
For a vertex subset $S$ of $G$, we write $G[S]$ for the subgraph of $G$ induced by $S$.
We write $\overline{G}$ for the \emph{complement} of $G$; that is, $u$ and $v$ are adjacent in $G$ if and only if they are not adjacent in $\overline{G}$.

We write $A(G)$ for the \emph{adjacency matrix} of $G$ over the binary field, that is, the $V(G)\times V(G)$ matrix over the binary field
such that the $(x,y)$-entry is one if $x\neq y$ and $x$ is adjacent to $y$ in $G$, and zero otherwise.
For an $X\times Y$  matrix $M$ and $X'\subseteq X$, $Y'\subseteq Y$, 
we write $M[X',Y']$ for the $X'\times Y'$ submatrix of $M$.

Let $P_n$ denote the path on $n$ vertices, and let $K_n$ denote the complete graph on $n$ vertices. The \emph{radius} of a tree is the minimum $r$ such that
there is a node having distance at most $r$ from every node.

For two $n$-vertex graphs $G$ and $H$ with fixed orderings $\{v_1,v_2,\ldots,v_n\}$ and $\{w_1, w_2, \ldots, w_n\}$ on their respective vertex sets, 
let 
$G\tri H$
be the graph with vertex set $V(G)\cup V(H)$ 
such that 
$(G\tri H)[V(G)]=G$, $(G\tri H)[V(H)]=H$, and 
 for all $i,j\in \{1,2,\ldots,n\}$, 
$v_iw_j\in E(G\tri H)$ if and only if $i\ge j$. See Figure~\ref{fig:HG} for an example.
An induced subgraph isomorphic to $G\tri H$ for some $G$ and $H$ is called a \emph{semi-induced half-graph} in \cite{NORS2019}.
        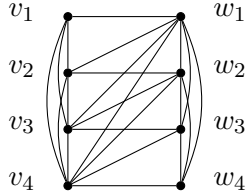
\begin{figure}
          \centering
          \tikzstyle{v}=[circle, draw, solid, fill=black, inner sep=0pt, minimum width=3pt]
        \begin{tikzpicture}
		\def \h {.75} \def \w {1.5}
          \foreach \i in  {1,...,4} {
           	\node[v,label={[left, xshift=-.3cm]$v_\i$}] at (0, -\i*\h) (v\i){};
		\node[v,label={[right, xshift=.3cm]$w_\i$}] at (\w, -\i*\h) (w\i){};
          }
	\foreach \i in {1,...,4}{%
	\foreach \j in {\i,...,4} {
	    \draw (v\j) -- (w\i);
	}}
	\foreach \i in {1,...,3}{%
		\FPeval{\result}{clip(\i+1)}
	    	\draw (v\i) -- (v\result);
		\draw (w\i) -- (w\result);
	}
	\draw (v1) to [bend right=16] (v3);%
	\draw (v2) to [bend right=16] (v4);
	\draw (v1) to [bend right=24] (v4);
	\draw (w1) to [bend left=16] (w3);
	\draw (w2) to [bend left=16] (w4);
	\draw (w1) to [bend left=24] (w4);
        \end{tikzpicture}
          \caption{The graph $K_4 \tri K_4$.}
          \label{fig:HG}
        \end{figure}
\subsection{Vertex-minors}
For a vertex $v$ of a graph $G$,
\emph{local complementation} at $v$
is an operation which results in a new graph $G*v$ on $V(G)$
such that
\[
  E(G*v)=E(G)\Delta\{xy: x, y\in N_G(v), x\neq y\}.\]
For an edge $uv$ of a graph $G$, the operation of \emph{pivoting} $uv$, denoted $G \pivot uv$, is defined as
$G\pivot uv:=G*u*v*u$.  See Oum~\cite{Oum2004} for further background and properties of local complementation and pivoting.
In particular, note that if $G$ is bipartite, then so is $G\pivot uv$.

A graph $H$ is \emph{locally equivalent} to $G$
if $H$ can be obtained from $G$ by a sequence of local complementations.
A graph $H$ is \emph{pivot equivalent} to $G$
if $H$ can be obtained from $G$ by a sequence of pivots.
A graph $H$ is a \emph{vertex-minor} of $G$ if
$H$ is an induced subgraph of a graph that is locally equivalent to $G$.
Finally, a graph $H$ is a \emph{pivot-minor} of $G$ if
$H$ is an induced subgraph of a graph that is pivot equivalent to $G$.

For a subset $S$ of $V(G)$, let $\cutrk_G(S)$ be the rank of the $S \times (V(G)\setminus S)$ submatrix of $A(G)$. This function is called the \emph{cut-rank} function of $G$.
It is easy to show that the cut-rank function is invariant under taking local complementations, again see Oum~\cite{Oum2004}. Thus we have the following fact.
\begin{lemma}\label{lem:monotone}
  If $H$ is a vertex-minor of $G$ and $X\subseteq V(G)$, then
  \[\rho_H(X\cap V(H))\le \rho_G(X).\]
\end{lemma}

The following lemmas will be used to find a long path.

	\begin{lemma}[Kim, Kwon, Oum, and Sivaraman~{\cite[Lemma 5.6]{KKOS2019}}]\label{lem:pivotpn1}
 The graph $\K_n\tri \S_n$ has a pivot-minor isomorphic to $P_{n+1}$.
	\end{lemma}
	\begin{lemma}[Kwon and Oum~{\cite[Lemma 2.8]{KO2014}}]\label{lem:pivotpn2}
 The graph $\S_n\tri \S_n$ has a pivot-minor isomorphic to $P_{2n}$.
	\end{lemma}

        \subsection{Rank-depth}
        We now review the notion of rank-depth, which was introduced by DeVos, Kwon, and Oum~\cite{DKO2019}.
A \emph{decomposition} of a  graph $G$  is 
a pair $(T,\sigma)$ of a tree $T$ and a bijection $\sigma$ from $V(G)$ to the set of leaves of $T$.
The \emph{radius} of a decomposition $(T,\sigma)$ 
is the radius of the tree $T$. 
For a non-leaf node $v \in V(T)$, the components of the graph $T - v$ give rise to a partition $\mathcal P_v$ of $V(G)$ by $\sigma$. 
The \emph{width} of $v$ is defined to be 
\[\max_{ \mathcal{P'} \subseteq \mathcal P_v } \cutrk_G \left( \bigcup_{X \in \mathcal{P'} } X \right).\]
The \emph{width} of the decomposition $(T,\sigma)$ is the maximum width of a non-leaf node of $T$.
We say that a decomposition $(T, \sigma)$ is a $(k,r)$-\emph{decomposition} of $G$
if the width is at most $k$ and the radius is at most $r$. 
The \emph{rank-depth} of a graph $G$ is the minimum integer $k$ such that 
$G$ admits a $(k, k)$-decomposition.
If $\abs{V(G)}<2$, then there is no decomposition and the rank-depth is zero.
Note that every tree in a decomposition has radius at least one and therefore
the rank-depth of a graph is at least one if $\abs{V(G)}\ge 2$.

By Lemma~\ref{lem:monotone}, it is easy to see the following.
\begin{lemma}[DeVos, Kwon, and Oum~\cite{DKO2019}]
  If $H$ is a vertex-minor of $G$, then the rank-depth of $H$ is at most the rank-depth of $G$.
\end{lemma}

	The next two lemmas will serve as a base case for induction in the proof of Theorem~\ref{thm:main1}.
	\begin{lemma}\label{lem:component}
	Let $G$ be a graph of rank-depth $m$.
	Then $G$ has a connected component of rank-depth at least $m-1$.
	\end{lemma}
	\begin{proof}
	If $m<2$, then it is trivial, as the one-vertex graph has rank-depth zero.
	Thus, we may assume that $m\ge 2$.
	
	Suppose for contradiction that 
	every connected component of $G$ has rank-depth at most $m-2$.
	Let $C_1, C_2, \ldots, C_t$ be the connected components of $G$.
	For each $i\in \{1, 2, \ldots, t\}$, 
	\begin{itemize}
	\item if $C_i$ contains at least two vertices, then 
	we take an $(m-2, m-2)$-decomposition $(T_i, \sigma_i)$ where $r_i$ is a node of $T_i$ having distance at most $m-2$ to every node of $T_i$, and
	\item if $C_i$ consists of one vertex, 
	then let $T_i$ be the one-node graph on $\{r_i\}$ and let $\sigma_i:V(C_i)\rightarrow \{r_i\}$ be the uniquely possible function.
	\end{itemize}
	
	We obtain a new decomposition $(T, \sigma)$ of $G$ 
	by taking the disjoint union of $T_i$'s and adding a new node $r$ and adding edges $rr_i$ for all $i\in \{1, 2, \ldots, t\}$. For every vertex $v$ of $G$, define $\sigma(v)=\sigma_i(v)$ if $v$ is a vertex of $C_i$.
	Then $(T, \sigma)$ has depth at most $m-1$ and width at most $m-2$.
	This contradicts the assumption that $G$ has rank-depth $m$. 
	
	We conclude that $G$ has a connected component of rank-depth at least $m-1$.
	\end{proof}

	The following lemma can be proven similarly to Lemma~\ref{lem:component}.
	For a graph $G$ of rank-depth $m$ and a non-empty vertex set $A$, 
	it is easy to check that $G-A$ has rank-depth at least $m-\abs{A}$, 
	and by Lemma~\ref{lem:component}, $G-A$ has a connected component of rank-depth at least $m-\abs{A}-1$.
	But, by a direct argument, we can guarantee that there is a connected component of $G-A$ of rank-depth at least $m-\abs{A}$.  We include the full proof for completeness.
	
	\begin{lemma}\label{lem:separation}
	Let $G$ be a graph of rank-depth $m$ and $A$ be a non-empty proper subset of $V(G)$.
	Then $G-A$ has a connected component of rank-depth at least $m-\abs{A}$.
  \end{lemma}
	\begin{proof}
	If $\abs{A}\ge m$, then any connected component has rank-depth at least zero.
	Thus, we may assume that $\abs{A}<m$. This implies that $m \ge 2$ as $A$ is non-empty.
	
	Suppose for contradiction that 
	every connected component of $G-A$ has rank-depth at most $m-\abs{A}-1$.
	Let $C_1, C_2, \ldots, C_t$ be the connected components of $G-A$.
	For each $i\in \{1, 2, \ldots, t\}$, 
	\begin{itemize}
	\item if $C_i$ contains at least two vertices, then 
	we take an $(m-\abs{A}-1, m-\abs{A}-1)$-decomposition $(T_i, \sigma_i)$ where $r_i$ is a node of $T_i$ having distance at most $m-\abs{A}-1$ to every node of $T_i$, and
	\item if $C_i$ consists of one vertex, 
	we set $T_i$ to be the one-node graph on $\{r_i\}$ and let $\sigma_i:V(C_i)\rightarrow \{r_i\}$ be the uniquely possible function.
	\end{itemize}
	
	We obtain a new decomposition $(T, \sigma)$ of $G$ 
	by taking the disjoint union of $T_i$'s and adding a new node $r$ and adding edges $rr_i$ for all $i\in \{1, 2, \ldots, t\}$, 
	and additionally appending $\abs{A}$ leaves to $r$ and assigning each vertex of $A$ to a distinct leaf with the map $\sigma$.
	 For every vertex $v$ of $G-A$, define $\sigma(v)=\sigma_i(v)$ if $v$ is a vertex of $C_i$.
	Then $(T, \sigma)$ has depth at most $m-\abs{A}$ and width at most $m-1$.
	Because $\abs{A}\ge 1$, this contradicts the assumption that $G$ has rank-depth $m$. 
	
	We conclude that $G-A$ has a connected component of rank-depth at least $m-\abs{A}$.
	\end{proof}
	
		\begin{lemma}\label{lem:merge}
	Let $m$ and $d$ be positive integers.
	Let $G$ be a graph with a vertex partition $(A,B)$ such that connected components of $G[A]$ and $G[B]$ have rank-depth at most $m$ and $\cutrk_G(A)\le d$.
	Then $G$ has rank-depth at most $m+d+1$.
	\end{lemma}
	\begin{proof}
	Let $C_1, \ldots, C_p$ be the connected components of $G[A]$, 
	and $D_1, \ldots, D_q$ be the connected components of $G[B]$.
	For each $i\in \{1, 2, \ldots, p\}$, 
	\begin{itemize}
	\item if $C_i$ contains at least two vertices, then 
	we take an $(m,m)$-decomposition $(T_i, \sigma_i)$ where $r_i$ is a node of $T_i$ having distance at most $m$ to every node of $T_i$, and
	\item if $C_i$ consists of one vertex, 
	then set $T_i$ as the one-node graph on $\{r_i\}$ and $\sigma_i:V(C_i)\rightarrow \{r_i\}$ as the uniquely possible function.
	\end{itemize}
	Similarly, we define $(F_j, \mu_j)$ for each $D_j$ where $f_j$ is a node of $F_j$ having distance at most $m$ to every node of $F_j$.
	
	Now, we obtain a new decomposition $(T, \sigma)$ of $G$ as follows.  Let $T$ be the tree obtained by taking the disjoint union of all of $T_i$'s and $F_j$'s, adding new vertices $x$ and $y$, an edge $xy$, 
	edges $xr_i$ for all $i\in \{1, \ldots, p\}$, and
	edges $yf_j$ for all $j\in \{1, \ldots, q\}$.
	Define $\sigma(v)=\sigma_i(v)$ if $v$ is a vertex of $C_i$, 
	and $\sigma(v)=\mu_j(v)$ if $v$ is a vertex of $D_j$.
	Then $(T, \sigma)$ has depth at most $m+2$ and width at most $m+d$.
	Because $d\ge 1$, $G$ has rank-depth at most $\max \{m+2, m+d\}\le m+d+1$. 
\end{proof}

\subsection{Rank-width}\label{sec:rankwidth}
We now review the definition of rank-width.
		A \emph{rank-decomposition} of a graph $G$ is a pair $(T, L)$ of a tree $T$ whose vertices each have degree either one or three, and a bijection $L$ from $V(G)$ to the set of leaves of $T$. 
		The \emph{width} of an edge $e$ of $T$ is the cut-rank in $G$ of the set of all leaves assigned to one of the components of $T-e$. 
		The \emph{width} of the rank-decomposition $(T, L)$ is the maximum width of an edge of $T$.
		Finally, the \emph{rank-width} of $G$ is the minimum width over all rank-decompositions of $G$. 
		Graphs with at most one vertex do not admit rank-decompositions and we define their rank-width to be zero.

\section{The proof}\label{sec:proof}

	We write $R(n;k)$ to denote the minimum number $N$ such
	that every coloring of the edges of $K_N$ with $k$ colors induces a monochromatic complete subgraph on $n$ vertices. The classical theorem of Ramsey~\cite{Ramsey1930} implies that $R(n;k)$ exists.

        The following lemma is well known. We include its proof for the sake of completeness.
	\begin{lemma}\label{lem:balance}
          Let $G$ be a graph of rank-width at most $q$
          and let $M\subseteq V(G)$.
          If $\abs{M}\ge 3k+1$ for a positive integer $k$,
          then there is a vertex partition $(X,Y)$ of $G$ such that 
	$\cutrk_H(X)\le q$ and
        $\min(\abs{M\cap X},\abs{M\cap Y})>k$.
      \end{lemma}
      \begin{proof}
        Suppose that there is no such vertex partition.
        Let $(T,L)$ be a rank-decomposition of width at most $q$.
        For each edge $uv$ of $T$, let us orient $e$ towards $v$
        if the component of $T-e$ containing $u$ has at most $k$ vertices in $L(M)$.
        By the assumption, every edge is oriented. Since $T$ is acyclic,
        there is a node $w$ of $T$ such that all edges of $T$ incident with $w$ are oriented towards $w$. But this implies that $\abs{M}\le 3k$, a contradiction.
      \end{proof}
	For a path $P$ with an endpoint $x$ and a graph $H$ and a non-empty subset of vertices $S\subseteq V(H)$, 
	we denote by $(P, x)+(H,S)$ the graph obtained from the disjoint union of $P$ and $H$
	by adding all edges between $x$ and $S$. We now prove our main proposition; Theorem~\ref{thm:main1} will follow quickly after.

	\begin{proposition}\label{prop:main}
	For all positive integers $a, b, t, q$,
	there exists an integer $\f(a,b,t,q)$ such that
	every graph of rank-width at most $q$ and rank-depth at least $\f(a,b,t,q)$ has
	a vertex-minor isomorphic to either $P_t$
	or $(P_a, x)+(H, S)$ where $x$ is an endpoint of $P_a$, $H$ is a connected graph
        of rank-depth at least $b$, and $S$ is a non-empty subset of $V(H)$.
	\end{proposition}
	\begin{proof}
	For all positive integers $b, t, q$, 
	we set 
	\begin{align*}
	f(1, b, t, q):=b+2,
	\end{align*} 
	and
	for $a\ge 2$, 
	we set
	\begin{align*}
          u&:=\max (3\cdot (2^q-1)+1, t-1),\\
          r&:=R( u+1 ;2^{a-1}),\\
          g_i&:=\begin{cases}
            b+q+2&\text{if }i=r,\\
            f(a-1,g_{i+1},t,q)& \text{if }i\in\{0,1,2,\ldots,r-1\},
            \end{cases}\\
          \f(a,b,t, q)&:=g_0.
	\end{align*}	
	We prove the proposition by induction on $a$.
  Let $G$ be a graph whose rank-depth is at least $f(a,b,t,q)$ and rank-width is at most $q$.
  If $a=1$, then it has a component $G'$ of rank-depth at least $b+1$ by
  Lemma~\ref{lem:component}. 
	Let $v\in V(G')$.
	By Lemma~\ref{lem:separation}, 
	$G'-v$ has a connected component $H$ of rank-depth at least $b$.
	So, $(G'[\{v\}], v)+(H, N_G(v)\cap V(H))$ is the second outcome.
	
	Thus, we may assume that $a\ge 2$. 
	Suppose that $G$ has no vertex-minor isomorphic to $P_t$.
        We claim that $G$ contains the second outcome.

        \begin{figure}
          \centering
          \tikzstyle{v}=[circle, draw, solid, fill=black, inner sep=0pt, minimum width=3pt]
        \begin{tikzpicture}
		\def \h {1.2} \def \w {1.2} 
          \foreach \i in  {1,2,3} {%
           	\node[v,label=left:{$A_\i$}] at (0, -\i*\h) (1\i){};
		\node[v] at (\w, -\i*\h) (2\i){};
		\node[v] at (2*\w, -\i*\h) (3\i){};
		\node[circle,draw,solid, fill=gray, inner sep=0pt,minimum width=5pt, label=right:{$x_\i$}] at (3*\w, -\i*\h) (4\i){};
          }
	\foreach \i in {1,...,3}{%
		\FPeval{\result}{clip(\i+1)}
		\foreach \j in {1,...,3}{
	    		\draw (\i\j) -- (\result\j);
	}}
	\draw (41) -- (12);%
	\draw (41) -- (42);
	\draw (42) -- (23);
	\draw (42) -- (33);
	\draw[thick] (4*\w+1.6,-2*\h) circle (1.7) {};
	\node[draw=none,label=center:{$H_3$}] at (4*\w+1.6,-2*\h) {};
	\draw (41) -- (4*\w,-2*\h+.6);
	\draw (42) -- (4*\w,-2*\h+.6);
	\draw (42) -- (4*\w,-2*\h-.6);
	\draw (43) -- (4*\w-.1,-2*\h-.1);
	\draw (43) -- (4*\w+.8,-2*\h-1.5);
        \end{tikzpicture}
          \caption{The graph $G_1[A_1 \cup \cdots\cup A_r \cup V(H_r)]$, when $r=3$ and $a=5$.}
          \label{fig:AxHS}
        \end{figure}
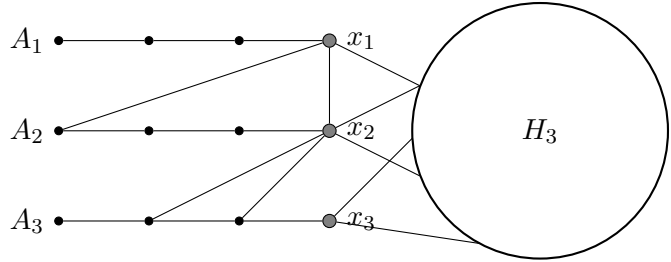
	Let $H_0:=G$. Observe that $H_0$ has rank-depth at least $\f(a,b,t,q)=g_0$.

	For $i\in \{1, 2, \ldots, r\}$,
	we recursively find tuples $(A_i, x_i, H_i, S_i)$ from $H_{i-1}$
	such that 
	\begin{itemize}
	\item $A_i$ is isomorphic to $P_{a-1}$ and $x_i$ is an endpoint of $A_i$, 
	\item $H_i$ is a connected graph of rank-depth at least $g_i$,
	\item $S_i$ is a non-empty subset of $V(H_i)$, and 
	\item $(A_i, x_i)+(H_i, S_i)$ is a vertex-minor of $H_{i-1}$.
	\end{itemize}
	Let $i\in \{1, 2, \ldots, r\}$ and assume that 
	$H_{i-1}$ is a given graph of rank-depth at least $g_{i-1}$.
	Then by the induction hypothesis,
	$H_{i-1}$ has a vertex-minor $(A_i, x_i)+(H_i, S_i)$ where
	$A_i$ is isomorphic to $P_{a-1}$, $x_i$ is an endpoint of $A_i$, $H_i$ is a connected graph of rank-depth at least $g_i$, and $S_i$ is a non-empty subset of $V(H_i)$.
	By the choice of functions $g_0, g_1, \ldots, g_r$, we can obtain the tuples for all $i\in \{1, 2, \ldots, r\}$.

	Observe that for $i< j$, no vertex in $V(A_i)\setminus \{x_i\}$ has a neighbor in $H_{i}$, 
	and therefore, the sequence of local complementations to obtain $(A_j, x_j)+(H_j, S_j)$ from $H_{j-1}$ does not change previous paths $A_1, \ldots, A_{j-1}$, 
	but may change the edges between $x_1$, $x_2$, $\ldots$, $x_{j-1}$.

	 By definition, 
	 $H_r$ is connected and has rank-depth at least $g_r$.
	Let $G_1$ be the graph obtained from $G$ by following the sequence of local complementations to obtain $(A_1, x_1)+(H_1, S_1), \ldots, (A_r, x_r)+(H_r, S_r)$. See Figure~\ref{fig:AxHS} for a depiction.
	Note that $G_1[V(H_i)\cup \{x_i\}]$ is connected for each $i$, as $H_i$ is connected, $S_i$ is non-empty, and we apply local complementations only inside $H_i$ to obtain later $H_j$'s.
	
	If for some $i\in \{1, 2, \ldots, r\}$, $N_{G_1}(x_i)\cap V(H_r)=\emptyset$, then 
	by taking a shortest path from $x_i$ to $V(H_r)$ in the graph $G_1[V(H_i)\cup \{x_i\}]$, 
	we can directly obtain the second outcome. 
	So, we may assume that 
	each of $\{x_1, x_2, \ldots, x_r\}$ has a neighbor in $V(H_r)$.
	
	Note that for $i<j$, only the endpoint $x_i$ in $A_i$ can have a neighbor in $A_j$ in $G_1$, 
	and therefore, there are $2^{a-1}$ possible ways of having edges between $A_i$ and $A_j$ in $G_1$.
	Since $r=R(u+1;2^{a-1})$, 
	by applying the theorem of Ramsey, 
	we deduce that there exists a subset $W\subseteq \{1, 2, \ldots, r\}$ of size $u+1$ 
	such that for all $i<j$ with $i,j\in W$,
        $\{ \ell: \text{the $\ell$-th vertex of $A_j$ is adjacent to $x_i$ in $G_1$}\}$
        are identical.
	
	If $x_i$ has a neighbor in $V(A_j-x_j)$ in $G_1$ for some $i<j$ with $i,j\in W$, 
	then $G_1$ has $\S_{u}\tri \S_{u}$ or $\S_{u}\tri \K_{u}$ as an induced subgraph.
	Since $u\ge t-1$,
	by Lemmas~\ref{lem:pivotpn1} and \ref{lem:pivotpn2},  
	$G_1$ contains a pivot-minor isomorphic to $P_t$, contradicting the assumption.
	So, for all $i<j$ with $i,j\in W$, 
	$x_i$ has no neighbors in $V(A_j-x_j)$.

	Note that $\{x_i:i\in W\}$ is an independent set or a clique in $G_1$.
	If it is an independent set, 
	then for some $i'\in W$, 
	we set 
	\begin{itemize}
		\item $G_2:=G_1$ and $W':=W\setminus \{i'\}$.
	\end{itemize}
	If $\{x_i:i\in W\}$ is a clique, then we choose a vertex $x_{i'}$ for some $i'\in W$ and locally complement at $x_{i'}$.  
	Then $\{x_i:i\in W\setminus \{i'\}\}$ becomes an independent set. 
	We set
	\begin{itemize}
	\item $G_2:=G_1*x_{i'}$ and $W':=W\setminus \{i'\}$.
	\end{itemize}
	Let $M:=\{x_i:i\in W'\}$ and $H:=G_2[V(H_r)\cup M\cup \{x_{i'}\}]$. 

	By definition, $H$ is locally equivalent to the graph $G_1[V(H_r)\cup M \cup \{x_{i'}\}]$. Thus, as the latter is connected, $H$ is also connected. Similarly, since $H_r=G_1[V(H_r)]$ has rank-depth at least $g_r$, 
        $H$ has rank-depth at least $g_r$.
	Also, note that $H$ has rank-width at most $q$ and $M$ is an independent set of size $u\ge 3\cdot (2^q-1)+1$ in $H$.
	Thus, by Lemma~\ref{lem:balance}, 
	$H$ admits a vertex partition $(X,Y)$ such that 
	$\abs{M\cap X}>2^q-1$, 
	$\abs{M\cap Y}>2^q-1$, and
	$\cutrk_H(X)\le q$.

	Since $H$ has rank-depth at least $g_r=b+q+2$ and $\cutrk_H(X)\le q$, 
	by Lemma~\ref{lem:merge},
        $H[X]$ or $H[Y]$ has a connected component of rank-depth at least $b+1$.  
	Without loss of generality, we assume that $H[X]$ has a connected component $Q$ of rank-depth at least $b+1$.
		
		Now, if $M\cap Y$ has a vertex $x_i$ that has no neighbor in $Q$, 
		then by taking a shortest path from $x_i$ to $Q$ in $H$, along with $A_i$, 
		we can find the second outcome.
		
		Thus, we may assume that in $H$, all vertices in $M\cap Y$ have a neighbor in $Q$.
		Since $\cutrk_H(X)\le q$,
		there are at most $2^q-1$ distinct non-zero rows in the matrix 
		$A(H)[M\cap Y, V(Q)]$.
		As $\abs{M\cap Y}\ge 2^q$, 
		by the pigeon-hole principle, 
		$H$ has  two vertices $x_{i_1}$ and $x_{i_2}$ in $M\cap Y$
                for some $i_1, i_2\in W'$ 
		that have the same neighborhood in $Q$.
		
		First assume that $x_{i_1}$ has exactly one neighbor in $Q$, say $w$.
		As $Q$ has rank-depth at least $b+1$, 
		$Q-w$ has a connected component $Q'$ having rank-depth at least $b$
                by Lemma~\ref{lem:separation}.
		Then 
		\[(G_2[V(A_{i_1})\cup \{w\}], w)+(Q', N_{G_2}(w)\cap V(Q'))\] is the required second outcome.
		So, we may assume that $x_{i_1}$ has at least two neighbors in $Q$.
		Let $w$ be a neighbor of $x_{i_1}$ in $Q$.
		
	Since $x_{i_1}$ and $x_{i_2}$ have the same neighborhood in $Q$ and they are not adjacent, 
	if we pivot $x_{i_2}w$, then the edges between $x_{i_1}$ and $N_H(x_{i_1})\cap V(Q)$ are removed and 
	$x_{i_2}$ becomes the unique neighbor of $x_{i_1}$ in $V(Q)\cup \{x_{i_2}\}$.
	Note that $G_2[V(Q)\cup \{x_{i_2}\}]$ is connected, and thus  
	$(G_2\pivot x_{i_2}w)[V(Q)\cup \{x_{i_2}\}]$ is also connected.
	As $Q$ has rank-depth at least $b+1$, 
	$(G_2\pivot x_{i_2}w)[V(Q)\cup \{x_{i_2}\}] - x_{i_2}$ has a connected component $Q'$
	that has rank-depth at least $b$.
	Then \[((G_2\pivot x_{i_2}w)[V(A_{i_1})\cup \{x_{i_2}\}], x_{i_2})+(Q', N_{G_2\pivot x_{i_2}w}(x_{i_2})\cap V(Q'))\]
	is the second outcome.
	This proves the proposition.
		\end{proof}

	Proposition~\ref{prop:main} implies the following result.
	\begin{theorem}\label{thm:boundedrw}
	For all positive integers $t$ and $q$, there exists an integer $F(t, q)$ such that every graph of rank-width at most $q$ and rank-depth at least $F(t,q)$ contains a vertex-minor isomorphic to $P_t$.	
	\end{theorem}
	\begin{proof}
          We take $F(t, q):=\f(t-1,1,t, q)$ where $\f$ is the function in Proposition~\ref{prop:main}. 
	\end{proof}

        A \emph{circle graph} is
        the intersection graph of chords on a circle.
        It is easy to see that $P_t$ is a circle graph.
	We can derive Theorem~\ref{thm:main1} by taking $q=\beta(P_t)$ from the following recent theorem.

	\begin{theorem}[Geelen, Kwon, McCarty, and Wollan~\cite{GKMW2019}]\label{thm:gridtheorem}
	For every circle graph $H$, there exists an integer $\beta(H)$ such that 
	every graph of rank-width more than $\beta(H)$ contains a vertex-minor isomorphic to $H$.
	\end{theorem}
        
	\begin{thm:main1}
	For every positive integer $t$, there exists an integer $N(t)$ such that every graph of rank-depth at least $N(t)$ contains a vertex-minor isomorphic to $P_t$.
	\end{thm:main1}
	\begin{proof}
          We take $N(t):=F(t, \beta(P_t))$ where $\beta$ is the function given in Theorem~\ref{thm:gridtheorem} and $F$ is the function from Theorem~\ref{thm:boundedrw}.
	If a graph has rank-width more than $\beta(P_t)$, then by Theorem~\ref{thm:gridtheorem}, 
	it contains a vertex-minor isomorphic to $P_t$.
	So, we may assume that a graph has rank-width at most $\beta(P_t)$.
	Then by Theorem~\ref{thm:boundedrw}, 
	it contains a vertex-minor isomorphic to $P_t$.
	\end{proof}
	
	\section{Pivot-minors}\label{sec:pivotminor}

	We can prove a stronger result on bipartite graphs, by slightly modifying the proof of Proposition~\ref{prop:main}.
	Suppose that a given graph $G$ is bipartite in the proof of Proposition~\ref{prop:main}.
	The only place that we have to apply local complementation
        instead of pivoting is when
	the set $\{x_i:i\in W\}$ is a clique, and we want to change it into an independent set.
	But if $G$ is bipartite, then  
	the obtained set $\{x_i:i\in W\}$ has no triangle, and so it is an independent set since $\abs{W}\ge 3$.
	Therefore, we can proceed only with pivoting.
	For bipartite graphs, we can use the following theorem due to Oum~\cite{Oum2004}, obtained as a consequence of the grid theorem for binary matroids~\cite{GGW07}.
	
	\begin{theorem}[Oum~\cite{Oum2004}]
	For every bipartite circle graph $H$, there exists an integer~$\gamma(H)$ such that 
	every bipartite graph of rank-width more than $\gamma(H)$ contains a pivot-minor isomorphic to $H$.
      \end{theorem}
      Thus we deduce the following theorem for bipartite graphs.
		\begin{thm:main2}
	For every positive integer $t$, there exists an integer~$N(t)$ such that every bipartite graph of rank-depth at least $N(t)$ contains a pivot-minor isomorphic to $P_t$.
  \end{thm:main2}
  
  Theorem~\ref{thm:main2} allows 
  us to obtain the following corollary for binary matroids, solving a special case of a conjecture of DeVos, Kwon, and Oum~\cite{DKO2019} on general matroids.
      We need a few terms to state the corollary. 
      The branch-depth of a matroid is defined 
      analogously to the definition of the rank-depth obtained by replacing the cut-rank function with the matroid connectivity function~\cite{DKO2019}.
      Let $F_t$ be the fan graph, that is  the union of $P_t$ with one vertex adjacent to all vertices of $P_t$, see Figure~\ref{fig:fan}. As usual, $M(F_t)$ denotes the cycle matroid of $F_t$.
      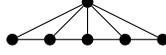
\begin{figure}
        \begin{center}
          \tikzstyle{v}=[circle,draw,fill=black,inner sep=0pt,minimum width=4pt]
        \begin{tikzpicture}
          \foreach \i in {1,2,3,4,5} {
            \draw (.5*\i,0) node [v] (v\i){};
          }
          \draw (1.5,.5) node[v] (v) {};
          \foreach \i in {1,2,3,4,5} {
            \draw (v)--(v\i);
          }
          \draw (v1)--(v5);
        \end{tikzpicture}
        \end{center}
        \caption{The fan graph $F_5$.}
                \label{fig:fan}
      \end{figure}
	\begin{corollary}\label{cor:main3}
	For every positive integer $t$, there exists an integer $N(t)$ such that every binary matroid of branch-depth at least $N(t)$ contains a minor isomorphic to $M(F_t)$.
        \end{corollary}
        \begin{proof}
          It is known \cite{Bouchet89, Oum2004} that the connectivity function of a binary matroid is equal to the cut-rank function
          of a corresponding bipartite graph, called a \emph{fundamental graph}.
          Furthermore for two binary matroids $M$ and $N$,
          if $N$ is connected, and a fundamental graph of $N$ is a pivot-minor of a fundamental graph of $M$,
      then either $N$ or $N^*$ is a minor of $M$, see Oum~\cite[Corollary 3.6]{Oum2004}.
      Since $(M(F_t))^*$ has a minor isomorphic to $M(F_{t-1})$, we deduce the corollary from Theorem~\ref{thm:main2}, because the path graph $P_{2t-1}$ is a fundamental graph of $M(F_t)$.
        \end{proof}

    We show that the class $\{K_n\tri K_n : n\ge 1\}$ has unbounded rank-depth, while for every positive integer $n$, $K_n\tri K_n$ has no pivot-minor isomorphic to $P_5$.
    It implies that contrary to Theorem~\ref{thm:main2}, 
    the class of graphs having no $P_n$ pivot-minor has unbounded rank-depth for each $n\ge 5$.
    
    Kwon and Oum~\cite[Lemma 6.5]{KO2019} showed that for every integer $n\ge 2$, $K_n\tri K_n$ contains a vertex-minor isomorphic to $P_{2n-2}$.
    Thus, $\{K_n\tri K_n : n\ge 1\}$ has unbounded rank-depth. 
    
    Now, we show that for $n\ge 1$, $K_n\tri K_n$ has no pivot-minor isomorphic to $P_5$.
    We prove a stronger statement that $K_n\tri K_n$ has no pivot-minor isomorphic to $K_{1,3}$.
    Dabrowski et al.~\cite{Konrad2018} characterized the class of graphs having no pivot-minor isomorphic to $K_{1,3}$ in terms of forbidden induced subgraphs.
    See Figure~\ref{fig:clawpm} for $\operatorname{bull}$, $W_4$, and $\overline{BW_3}$.
       
    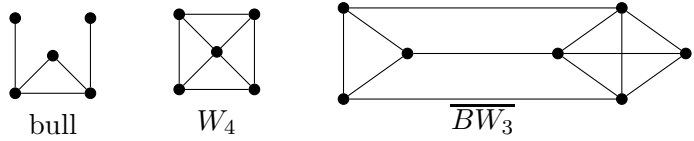
\begin{figure}
  \begin{center}
    \tikzstyle{v}=[circle,draw,fill=black,inner sep=0pt,minimum width=4pt]
  $\quad\quad$
  \begin{tikzpicture}
    \draw (0,0) node [v] (v1){};
    \draw (1,0) node [v] (v2){};
    \draw (1,1) node [v] (v3){};
    \draw (0,1) node [v] (v4){};
    \draw (.5,.5) node[v] (v) {};
    \draw (v4)--(v1)--(v)--(v2)--(v3);
    \draw (v1)--(v2);
    \draw (0.5,0) node[label=below:bull]{};
  \end{tikzpicture}
  $\quad\quad$
  \begin{tikzpicture}
    \draw (0,0) node [v] (v1){};
    \draw (1,0) node [v] (v2){};
    \draw (1,1) node [v] (v3){};
    \draw (0,1) node [v] (v4){};
    \draw (.5,.5) node[v] (v) {};
    \foreach \i in {1,2,3,4} {
      \draw (v)--(v\i);
      }
      \draw (v1)--(v2)--(v3)--(v4)--(v1);
      \draw (0.5,0) node[label=below:$W_4$]{};
  \end{tikzpicture}
  $\quad\quad$
  \begin{tikzpicture}
    \begin{scope}[xshift=-1.5cm]
      \foreach \i in {1,2} {
      \draw (120*\i:.7) node[v](a\i){};
    }
      \foreach \i in {3} {
      \draw (0:.5) node[v](a\i){};
    }
    \draw (a1)--(a2)--(a3)--(a1);
    \end{scope}
    \begin{scope}[xshift=1.5cm]
      \draw (0:1.2) node[v] (b4){};
      \foreach \i in {1,2} {
      \draw (180-120*\i:.7) node[v](b\i){};
      \draw (a\i)--(b\i)--(b4);
    }
      \foreach \i in {3} {
      \draw (180:.5) node[v](b\i){};
      \draw (a\i)--(b\i)--(b4);
    }
    \draw (b1)--(b2)--(b3)--(b1);
    \end{scope}
    \draw (-90:.4) node[label=below:$\overline{BW_3}$]{};
  \end{tikzpicture}
  \end{center}
  \caption{The graphs $\operatorname{bull}$, $W_4$, and $\overline{BW_3}$.}
          \label{fig:clawpm}
\end{figure}
    \begin{theorem}[Dabrowski et al.~\cite{Konrad2018}]\label{thm:dabrowski}
    A graph has a pivot-minor isomorphic to $K_{1,3}$ if and only if it has an induced subgraph isomorphic to one of $K_{1,3}$, $P_5$, $\operatorname{bull}$, $W_4$, and $\overline{BW_3}$.
    \end{theorem}
    \begin{lemma}\label{lem:knkn}
    For $n\ge 1$, $K_n\tri K_n$ has no induced subgraph isomorphic to one of $K_{1,3}$, $P_5$, $\operatorname{bull}$, $W_4$, and $\overline{BW_3}$.
    \end{lemma}
    \begin{proof}
    As the maximum size of an independent set in $K_n\tri K_n$ is $2$, 
    $K_n\tri K_n$ has no induced subgraph isomorphic to one of $K_{1,3}$, $P_5$, and $\operatorname{bull}$.
    
    Also $K_n\tri K_n$ has no induced cycle of length $4$ because such a cycle should contain two vertices in each $K_n$
    but the edges between two $K_n$'s have no induced matching of size $2$.
    Therefore, 
    it has no induced subgraph isomorphic to $W_4$ or $\overline{BW_3}$.
    \end{proof}
    
    By Theorem~\ref{thm:dabrowski} and Lemma~\ref{lem:knkn}, 
    $K_n\tri K_n$ has no pivot-minor isomorphic to $K_{1,3}$, and to $P_5$.
    Thus, for all $n\ge 5$, the class of graphs having no $P_n$ pivot-minor includes $\{K_n\tri K_n : n\ge 1\}$, which has unbounded rank-depth.
    It may be interesting to see whether every graph with sufficiently large rank-depth contains either $P_n$ or $K_n\tri K_n$ as a pivot-minor.
    We leave it as an open question.
    \begin{question}
    Does there exist a function $f$ such that for every $n$, every graph with rank-depth at least $f(n)$ contains a pivot-minor isomorphic to $P_n$ or $K_n\tri K_n$?
    \end{question}

        \section{Concluding remarks}\label{sec:remark}
 
        \subsection{Linear $\chi$-boundedness}\label{subsec:chi}

        We define linear rank-width.
	For an ordering $(v_1, v_2, \ldots, v_n)$ of the vertex set of a graph $G$, 
	its \emph{width} is defined as the maximum of $\cutrk_G(\{v_1, \ldots, v_i\})$ for all $i\in \{1, 2, \ldots, n-1\}$, 
	and the \emph{linear rank-width} of $G$ is defined as the minimum width of all orderings of $G$.
  If $\abs{V(G)}<2$, then the linear rank-width of $G$ is defined as $0$.

  Graphs of bounded rank-depth have bounded linear rank-width, which was already known through the notions of shrub-depth and linear clique-width~\cite{GHNOO2017}. Kwon and Oum~\cite{KO2019a} proved it directly as follows.

  \begin{proposition}[Kwon and Oum~\cite{KO2019a}]\label{prop:lrwbound}
 Every graph of rank-depth $k$ has linear rank-width at most~$k^2$.
\end{proposition}

        We write $\chi(G)$ to denote the chromatic number of $G$
        and $\omega(G)$ to denote the maximum size of a clique of $G$.
        A class $\mathcal C$ of graphs  is \emph{$\chi$-bounded}
        if there is a function $f$ such that
        $\chi(H)\le f(\omega(H))$ for all induced subgraphs $H$ of a graph in $\mathcal C$.
        In addition, if $f$ can be taken as a polynomial function, then
        $\mathcal C$ is \emph{polynomially $\chi$-bounded}.
        If $f$ can be taken as a linear function, then
        $\mathcal C$ is \emph{linearly $\chi$-bounded}.

        \begin{proposition}[Ne\v{s}et\v{r}il, Ossona de Mendez, Rabinovich, and Siebertz~\cite{NORS2019}]\label{prop:lrw}
          For every positive integer $r$, there exists an integer $c(r)$ such that
          for every graph $G$ of linear rank-width at most $r$, 
          \[
            \chi(G)\le c(r) \, \omega(G).
          \]
        \end{proposition}

      By combining Proposition~\ref{prop:lrwbound} and Proposition~\ref{prop:lrw}, we can prove the following, which answers a previous question by Kim, Kwon, Oum, and Sivaraman~\cite{KKOS2019}.
      \begin{theorem}\label{thm:linearchi}
        For every positive integer $t$, 
        the class of graphs with no vertex-minor isomorphic to $P_t$ is
        linearly $\chi$-bounded.
      \end{theorem}

		We remark that there is an alternative way to prove Theorem~\ref{thm:linearchi} without using linear rank-width. First, DeVos, Kwon, and Oum~\cite[Lemma 4.10]{DKO2019} showed that if a graph has rank-depth $k$,
		then it has an $(a,k)$-shrubbery where \[ a= (1+o(1))2^{(2^{2k+1}(2^{2k+2}-1)+1)k/2}.\]  (Please see \cite{DKO2019} for the definition of an $(a,k)$-shrubbery.)
		Lemma 2.16~of Ne\v{s}et\v{r}il, Ossona de Mendez, Rabinovich, and Siebertz~\cite{NORS2019} states that every class of bounded shrub-depth can be partitioned into bounded number of vertex-disjoint induced subgraphs, each of which is a cograph. 
		Its (short and easy) proof shows that a graph with an $(a,k)$-shrubbery can be partitioned into at most $a$ vertex-disjoint induced subgraphs, each of which is a cograph. 
		Since cographs are perfect, we deduce that if $G$ has rank-depth at most $k$, then $\chi(G)\le \omega(G)(1+o(1))2^{(2^{2k+1}(2^{2k+2}-1)+1)k/2}$.

        \subsection{When does the class of $\mathcal H$-vertex-minor-free graphs have
          bounded rank-depth?}
        For a set $\mathcal H$ of graphs, we say that $G$ is \emph{$\mathcal H$-minor-free}
        if $G$ has no minor isomorphic to a graph in $\mathcal H$,
        and $G$ is \emph{$\mathcal H$-vertex-minor-free}
        if $G$ has no vertex-minor isomorphic to a graph in $\mathcal H$.
        Robertson and Seymour~\cite{RS1991} showed that $\mathcal H$-minor-free graphs
        have bounded tree-width if and only if $\mathcal H$ contains a planar graph.
        As an analogue, Geelen, Kwon, McCarty, and Wollan~\cite{GKMW2019} showed
        that $\mathcal H$-vertex-minor-free graphs have bounded rank-width
        if and only if $\mathcal H$ contains a circle graph. Interestingly, Theorem~\ref{thm:main1} allows us to characterize the classes $\mathcal H$ such that $\mathcal H$-vertex-minor-free graphs have bounded rank-depth. This is due to the following theorem; the equivalence between (a) and (b) was shown by Kwon and Oum~\cite{KO2013}
        and the equivalence between (a) and (c) was shown by  Adler, Farley, and Proskurowski~\cite{AFP2013}. 
        \begin{theorem}[Kwon and Oum~\cite{KO2013};
          Adler, Farley, and Proskurowski~\cite{AFP2013}]\label{thm:lrw1}
          Let $H$ be a graph. The following are equivalent.
          \begin{enumerate}[(a)]
          \item $H$ has linear rank-width at most one.
          \item $H$ is a vertex-minor of a path.
          \item $H$ has no vertex-minor isomorphic to $C_5$, $N$, or $Q$ in Figure~\ref{fig:cnq}.
          \end{enumerate}
        \end{theorem}
        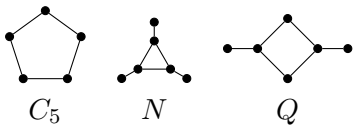
\begin{figure}
          \centering
          \tikzstyle{v}=[circle, draw, solid, fill=black, inner sep=0pt, minimum width=3pt]
          \begin{tabular}{ccc}
        \begin{tikzpicture}[scale=.5]
          \foreach \i in  {0,1,2,3,4} {
            \node [v] at (72*\i+90:1) (v\i){};
          }
          \draw (v0)--(v1)--(v2)--(v3)--(v4)--(v0);
        \end{tikzpicture}
            &
        \begin{tikzpicture}[scale=.5]
          \foreach \i in  {0,1,2} {
            \node [v] at (120*\i+90:1) (v\i){};
            \node [v] at (120*\i+90:.5) (w\i){};
            \draw(v\i)--(w\i);
          }
          \draw (w0)--(w1)--(w2)--(w0);
        \end{tikzpicture}
            &
        \begin{tikzpicture}[scale=.8]
          \foreach \i in  {0,2} {
            \node [v] at (90*\i+90:.5) (w\i){};
          }
          \foreach \i in  {1,3} {
            \node [v] at (90*\i+90:1) (v\i){};
            \node [v] at (90*\i+90:.5) (w\i){};
            \draw(v\i)--(w\i);
          }
          \draw (w0)--(w1)--(w2)--(w3)--(w0);
        \end{tikzpicture}
            \\
            $C_5$&$N$&$Q$
                       \end{tabular}
          \caption{Obstructions for being a vertex-minor of a path.}
          \label{fig:cnq}
        \end{figure}
        We define linear rank-width in the next subsection. Here, we only need the fact that
        linear rank-width does not increase when we take vertex-minors
        and 
        that paths have linear rank-width $1$ and arbitrary large rank-depth to deduce the following corollary from Theorems~\ref{thm:lrw1} and \ref{thm:main1}.
        \begin{corollary}
          Let $\mathcal H$ be a set of graphs. Then the following are equivalent.
          \begin{enumerate}[(a)]
          \item The class of $\mathcal H$-vertex-minor-free graphs has bounded rank-depth.
          \item $\mathcal H$ contains a graph of linear rank-width at most one.
          \item $\mathcal H$ contains a graph with no vertex-minor isomorphic to $C_5$, $N$, or $Q$.
          \end{enumerate}
        \end{corollary}

\paragraph*{Acknowledgement.}
The authors would like to thank anonymous reviewers for their helpful suggestions.

\end{document}